\documentclass[11pt]{amsart}

\usepackage{epsfig,amsmath,amsfonts,latexsym}
\usepackage[abs]{overpic}
\usepackage{color}
 \usepackage{palatino}
 \usepackage{hyperref}
 
\usepackage{tikz-cd}
\usepackage[mathscr]{euscript}

\newtheorem{question}{Question}
\newtheorem{conjecture}{Conjecture}
\newtheorem{theorem}{Theorem}

\newtheorem{lemma}[theorem]{Lemma}
\newtheorem{corollary}[theorem]{Corollary}
\newtheorem{step}{Step}

\theoremstyle{definition}
\newtheorem{definition}[theorem]{Definition}

\theoremstyle{remark}
\newtheorem{remark}[theorem]{Remark}

\def\R{\mathbb{R}}
\def\Z{\mathbb{Z}}

\def\N{\mathbb{N}}
\def\Q{\mathbb{Q}}

\numberwithin{theorem}{section}
\theoremstyle{plain}

\begin{document}

\title []{A NOTE ON EMBEDDINGS OF 3-MANIFOLDS IN SYMPLECTIC 4-MANIFOLDS}

\author{Anubhav Mukherjee}

\address{School of Mathematics \\ Georgia Institute
of Technology \\  Atlanta  \\ Georgia}

\email{anubhavmaths@gatech.edu}


\begin{abstract}
We show that any closed oriented 3-manifold can be topologically embedded in some simply-connected closed symplectic 4-manifold, and that it can be made a smooth embedding after one stabilization. As a corollary of the proof we show that the homology cobordism group is generated by Stein fillable 3-manifolds. We also find  obstructions on smooth embeddings: there exists 3-manifolds which cannot smoothly embed in a way that appropriately respect orientations in any symplectic manifold with weakly convex boundary. This embedding obstruction can also be used to detect exotic smooth structures on 4-manifolds.
\end{abstract}

\maketitle

\section{Introduction}

The embedding of 3-manifolds in higher dimensional space has always been a fascinating problem. Whitney's embedding theorem \cite{whitney} says that every closed oriented 3-manifold smoothly embeds in $\R ^6$. Hirsch improved this result by proving \cite{hirsch}  that every 3-manifold can be smoothly embedded in $S^5$. Meanwhile, Lickorish \cite{Lickorish} and Wallace \cite{Wallace} proved that every 3-manifold can be smoothly embedded in some 4-manifold, and in fact, a generalization of their arguments shows that every 3-manifold can be smoothly embedded in the connected sum of copies of $S^2\times S^2$. Freedman proved \cite{freedman} that all integer homology 3-spheres can be embedded topologically, locally flatly in $S^4$. On the other hand, the Rokhlin invariant $\mu$ and Donaldson's diagonalization theorem \cite{donaldosn} show that some integer homology spheres cannot smoothly embed in $S^4$.  Now, one can ask: Does there exists a compact 4-manifold in which all 3-manifolds embed? Shiomi \cite{shiomi} gave a negative answer to this question. Aceto, Golla, and Larson \cite{golla17} studied the problem of embedding 3-manifolds in spin 4-manifolds. Symplectic manifolds form a very interesting class of 4-manifolds. Etnyre, Min, and the author conjectured the following in \cite{mukherjee19}.
\begin{conjecture}\label{conjecture}
Every closed, oriented smooth 3-manifolds smoothly embed in a symplectic 4-manifold.
\end{conjecture}
\noindent
 For example, notice that if $Y$ is obtained by doing integer surgery on a knot in $S^3$ then $Y$ can have a   smooth oriented embedding in $\mathbb{CP}^2 \# \overline{\mathbb{CP}^2}$ or $\mathbb{CP}^2\#_2 \overline{\mathbb{CP}^2}$ depending on whether the surgery coefficient is odd or even. There does not seem to be an analogous result for links. 

\subsection{Embeddings in symplectic manifolds}
While we cannot resolve the above conjecture, we can show the existence of topological embeddings and of smooth embeddings after stabilization.

\begin{theorem}\label{ttop} Given a closed, connected, oriented 3-manifold $Y$ there exists a simply-connected symplectic closed 4-manifold $X$ such that $Y$ can be embedded topologically, locally flatly (i.e. it has collar neighbourhood) in $X$. This embedding can be made a smooth embedding after one stabilization, that is $Y$ can smoothly embed in $X\#(S^2 \times S^2)$.

\end{theorem}


When smooth embedding in symplectic manifolds do exist, they can imply interesting things about topology. For example, 
recall that a closed oriented rational homology sphere is called an $L$-\textit{space} if its Heegaard Floer homology group is ``as simple as possible,'' as specified in Section~\ref{heegard}. Embeddings of $L$-spaces in symplectic manifolds are constrained as follows. 

\begin{theorem}\label{psep} 
If an $L$-space $Y$ smoothly embeds in a closed symplectic 4-manifold $X$ then it has to be separating. Moreover, if $X = X_1 \cup _Y X_2$ then one of the $X_i$ has to be a negative-definite 4-manifold.

\end{theorem}






\begin{remark}
 Ozsv\'ath and Szab\'o \cite{oss} have established the above result for separating $L$-spaces in a symplectic manifolds. So the main new content of the above theorem is to show that $L$-spaces cannot be embedded as non-separating hypersurfaces in symplectic manifolds. The proof we give of this was inspired by Agol and Lin's work on hyperbolic $4$-manifolds \cite{agol2018hyperbolic}. 
\end{remark}

Theorem~\ref{psep} gives rise to a very interesting question.
\begin{question}\label{1} Does every $L$-space bound a definite 4-manifold?

\end{question}

\begin{remark}

 Notice that if the Conjecture~\ref{conjecture} proposed by Etnyre, Min, and the author \cite{mukherjee19} is true then the above question has a positive answer.
\end{remark}
In Section~\ref{cobord}, after we prove Theorem~\ref{psep}, we will discuss a strategy in Remark~\ref{strategy} to give a negative answer to Question~\ref{1}, and such an answer would in turn give a counterexample to the conjecture that closed $3$-manifolds can be smoothly embedded in symplectic $4$-manifolds.

\subsection{Cobordisms and symplectic structures}
We say that a closed oriented 3-manifold is \textit{Stein fillable} if there is a Stein fillable contact structure on it, whose definition is deferred until Section~\ref{dstein}. Not all 3-manifolds are Stein fillable. Work of Eliashberg \cite{etight} and Gromov \cite{gtight} proved that Stein fillable contact structures are always tight. Lisca \cite{lisca} gave the first example  of a non-Stein fillable manifold, and Etnyre and Honda \cite{eh} improved the result by showing the existence of a 3-manifold without a tight contact structure.

Let $Y_0$ and $Y_1$ be smooth, oriented, closed $n$-manifolds. A \textit{cobordism} from $Y_0$ to $Y_1$ is a compact $(n+1)$-dimensional smooth, oriented, compact manifold $W$ with $\partial W= -Y_0\sqcup Y_1$. We say $Y_0$ and $Y_1$ are \textit{$R$-homology cobordant}, written $Y_0\sim Y_1$, if $H_*(W,Y_i;R)=0$ for $i=0,1$. We call this \textit{integral homology cobordism} when $R=\Z$ and \textit{rational homology cobordism} when $R=\Q$. This is an equivalence relation, so one can define
\[
\Theta^n_R \ = \ \{ Y \textit{is a smooth oriented closed n-manifold with} \   H_*(Y;R)=H_*(S^n;R)\}\ / \sim 
\]
where $R$ is a fixed commutative ring. We give $\Theta_R^n$ the structure of a group where summation is given by the connected sum operation. The zero element of this group is given by the class of $S^n$, and the inverse of the class of $[ Y]$ is given by the class of $Y$ with reversed orientation. 
In low-dimensional topology the study of $\Theta^3_{\Z}$ and $\Theta^3_{\Q}$ are of special interest.

\begin{theorem}\label{cemb} The homology cobordism groups $\Theta^3_\Z$ and $\Theta^3_\Q$ are generated by Stein fillable 3-manifolds.
\end{theorem}

\begin{remark}
  It is not known whether $\Theta^3_{\Q}$ is generated by $L$-spaces or not. Nozaki, Sato and Taniguchi \cite{masaki} proved that $\Sigma (2,3,11) \#_2(- \Sigma(2,3,5))$ does not bound definite 4-manifold. If we can find an $L$-space $Y$ which is rationally cobordant to $\Sigma (2,3,11)$, then $Y\#_2(-\Sigma(2,3,5))$ cannot bounds a definite 4-manifold. Since $Y\#_2(-\Sigma (2,3,5))$ is an $L$-space, Theorem~\ref{psep} says this manifold cannot be smoothly embedded in any symplectic 4-manifold. Conversely, if all 3-manifolds embed in some symplectic 4-manifold, then $\Theta^3_{\Q}$ is not generated by $L$-spaces. So finally we can ask the following question. 
 \end{remark}
 
 \begin{question}
  Is $\Sigma(2,3,11)$ rationally cobordant to some $L$-space?
 \end{question}

For a closed oriented 3-manifold $Y$, $H_3(Y;\Z)$ is canonically isomorphic to $\Z$. So a map $f:Y_0 \rightarrow Y_1$ induces a homomorphism on the top-dimensional homology group, $f_*: \Z \rightarrow \Z$. The degree of $f$ is $f_*(1)\in \Z$. 

\begin{theorem}\label{deg}
Given any 3-manifold $Y$ there exists a Stein fillable 3-manifold $Y'$ and a degree one map $f: Y' \to Y$.
\end{theorem}

Although we cannot give a complete answer about smooth embeddings of 3-manifolds in closed symplectic 4-manifolds, we can find obstructions to smooth embeddings in compact symplectic 4-manifolds with convex boundary. There is an ambiguity when we think about smooth embeddings of a 3-manifold $Y$ in a smooth 4-manifold $X$. As oriented manifold $Y$ is very different from $-Y$: for example the Poincar\'{e} homology sphere with positive orientation bounds a negative-definite 4-manifold but the Poincar\'{e} homology sphere with negative orientation does not. On the other hand, if $Y$ smoothly embeds in $X$ then the boundary of a small neighbourhood of $Y$ is $-Y\sqcup Y$, so $-Y$ smoothly embeds in $X$ as well. But we can fix this issue in terms of cobordisms. 

\begin{definition} We call a smooth embedding of $Y$ in a cobordism $W$ from $Y_0$ to $Y_1$ an {\em oriented\ cobordism\ embedding} if $Y$ is either non-separating or $Y$ separates $W$ into $W_1\sqcup W_2$ such that  $Y$ as an oriented manifold is a boundary component of $W_1$, and all other components of $\partial W_1$ (if they exist) are part of $Y_0$.

\end{definition} 


\begin{theorem}\label{lemb}
If an $L$-space $Y$ does not bound a negative-definite 4-manifold then $Y$ cannot have an oriented  cobordism embedding in any symplectic 4-manifold with weakly convex boundary.
\end{theorem}

\begin{remark}
 There are many such $L$-spaces, for instance the Poincar\'{e} homology sphere with negative orientation and $r$-surgery for $r\in [9,15)$ on the pretzel knot $P(-2,3,7)$ in $S^3$ \cite{tosun17, li} (the latter $L$-spaces are in fact hyperbolic). It was already known that these manifolds are not Stein fillable \cite{tosun17, li, lisca}. Here, we proved that in addition to not being weakly fillable they cannot even have a smooth oriented cobordism embedding in any weak filling of any 3-manifold.
\end{remark}

 \begin{corollary}\label{YxI}
 If $Y'$ admits a  weakly fillable contact structure then any $L$-space $Y$ which does not bound negative-definite 4-manifolds cannot have any smooth oriented cobordism embedding in  $Y'\times I$.
 \end{corollary}

 The difference between smooth and topological embeddings can be used to detect exotic structures on compact manifolds. If we find two homeomorphic 4-manifolds such that a 3-manifold embeds smoothly in one but not the other then they are not diffeomorphic, i.e. they are an exotic pair. The upcoming Corollary~\ref{cex} was first proved by Akbulut \cite{a91} and since then by many others, but we will give an alternative proof that follows from the study of embeddings of $3$-manifolds into $4$-manifolds. 
 
 \begin{corollary}\label{cex}
 
 There exists compact 4-manifolds with boundary $X$ and $X'$ such that $b_2(X)=b_2(X')=1$ that are homeomorphic but not diffeomorphic.
 \end{corollary}

 
We now turn to studying when a 3-manifold has a Stein filling for some contact structure, and start by discussing obstructions.  The \textit{Rokhlin  invariant} $\mu : \Theta^3_{\Z}\rightarrow \Z /2$ is defined as $\mu(Y)= \sigma (W)/8  \pmod{2}$, where $W$ is any compact, spin 4-manifold with boundary $Y$ and $\sigma(W)$ is its signature. This invariant $\mu$ is an invariant under homology cobordism. The Brieskorn homology sphere $\Sigma(2,3,7)$ cannot bound a $\Z HB^4$ since its Rokhlin invariant $\mu$ is 1. So any 3-manifold $Y$ that is homology cobordant to $\Sigma(2,3,7)$ cannot have an integer homology ball ( $\Z HB^4$) as a Stein filling. But Fintushel and Stern \cite{FS84} proved that $\Sigma (2,3,7)$ bounds a rational homology ball ($\Q HB^4$). So one can ask if $\Sigma (2,3,7)$ has a $\Q HB^4$ as a Stein filling. The following lemma is well-known and can be proven easily by looking at the long exact homology sequence.
\begin{lemma}\label{3handle}
If a $\Z HS^3$ bounds a $\Q HB^4$ which is not a $\Z HB^4$ then it must have a 3-handle. 
\end{lemma}

 Recall that we call a integral homology cobordism from $Y_0$ to $Y_1$ a $\Z$-\textit{ribbon cobordism} if this integer homology cobordism is achieved by attaching 1- and 2-handles and no 3-handle to $Y_0\times [0,1]$ along $Y_0\times \{1\}$. We also indicate such a cobordism by saying $Y_0$ is \textit{ribbon cobordant} to $Y_1$. Note that this relation is a partial ordering on 3-manifolds and not necessarily a symmetric relation. (We can similarly define $\Q$-ribbon homology cobordism.)

  Lemma~\ref{3handle} implies that if $\Sigma(2,3,7)$ bounds a $\Q HB^4$ then it cannot be Stein as every handle decomposition has a 3-handle which contradicts a result of Eliashberg \cite{yasha91}. From the previous discussion we can see that the same conclusion is true for any 3-manifold $Y$ that is integer homology cobordant to $\Sigma(2,3,7)$. So it is natural to ask if there exists a 3-manifold $Y$ that  is rationally cobordant to $\Sigma(2,3,7)$ and it bounds a rational homology Stein ball. we know that $S^3$ is rationally cobordant to $\Sigma(2,3,7)$ \cite{FS84}. But such a cobordism must have a 3-handle. So a modified question would be: Does there exist a 3-manifold $Y$ such that there is a rational ribbon homology cobordism from $\Sigma(2,3,7)$ to $Y$ and $Y$ bounds a rational Stein ball? 
 
 \begin{theorem}\label{cstein}
 If $X$ is an oriented compact 4-manifold with connected boundary $\partial{X}= Y$ and $b_1(X)=0$ then there exists a Stein 4-manifold $X'$ with boundary $\partial X'=Y'$  such that there is a rational ribbon homology cobordism from $Y$ to $Y'$ and $b_2(X)=b_2(X')$.
\end{theorem}

As discussed above, not every smooth filling $X$ of a $3$-manifold $Y$ can be given a Stein structure, or indeed there are $Y$ that do not even admit any Stein fillings.  But we can ask if there is a ribbon rational homology cobordism from $Y$ to a manifold $Y'$ that has a symplectic filing $X'$ so that $X'$ shares some of the algebraic properties of $X$. For example if we let 
\[
b_2^F(Y)=\min\{b_2(X) \mid \partial X = Y\},
\]
then we can ask the following.
\begin{question}\label{Q2}
Let $Y$ be a 3-manifold. Is there a ribbon rational homology cobordism to $Y'$ such that $b_2^F(Y')=b_2^F(Y)$ and $Y'$ has a Stein filling which realized $b_2^F(Y')$?
\end{question}
While this is an interesting question that we cannot answer in full generality, if we restrict to the case of fillings with $b_1=0$, then the above Theorem~\ref{cstein} says the answer is yes. Or more explicitly, if we set
\[
b_2^{F,0}(Y)= \min\{b_2(X) | \partial X = Y, b_1(X)=0\},
\]
then Theorem~\ref{cstein} says that given any $4$-manifold $X$ with boundary $Y$ there is a ribbon rational homology cobordism from $Y$ to a manifold $Y'$ such that $Y'$ has a Stein filling and $b_2^{F,0}(Y')=b_2^{F,0}(Y)$. In general, we end by asking, given a $4$-manifold $X$ with boundary $Y$, when can one find an integral or rational cobordism from $Y$ to $Y'$ and a symplectic (or Stein) filling $X'$ of $Y'$ such that various algebraic invariants of $X'$ agree with those of $X$? For example we could ask for the same homology, signature, and fundamental group, among other things. Although we do not have an answer for every 3-manifolds but for $\Q HS^3$ we can answer the above Question~\ref{Q2}.

\begin{theorem}\label{AQ2}
 If $X$ is a compact 4-manifold with connected boundary $\partial X= Y$ a $\Q HS^3$, then there exists a Stein 4-manifold $X'$ with boundary $\partial X' = Y'$ such that the intersection form of $X$ is isomorphic to the intersection form of $X'$ and there is a rational ribbon homology cobordism from $Y$ to $Y'$.

\end{theorem}

\begin{corollary}
There exists a 3-manifold $Y'$ and a rational ribbon homology cobordism from $\Sigma(2,3,7)$ to  $Y'$ such that $Y'$ has a rational ball Stein filling.
\end{corollary}
\begin{remark}
 The above corollary is true if we replace $\Sigma(2,3,7)$ with any 3-manifold which bounds a rational ball. A large class of such manifolds is provided by Akbulut and Larson \cite{AL18}.
\end{remark}

We end with a slightly more technical result, on which several of the above results depend. 
\begin{theorem} \label{c1} Given any closed oriented 3-manifold $Y$ there exists a Stein fillable 3-manifold $Y'$ and is a $\Z$ ribbon homology cobordism $W$ from $Y$ to $Y'$ which is obtained from $Y\times [0,1]$ by attaching  a single pair of  algebraically cancelling 1- and 2-handle. Moreover, this is an invertible cobordism, that is there is a cobordism $W'$ from $Y'$ to $Y$ such that $W\cup_{Y'} W'$ is diffeomorphic to $Y\times [0,1]$. In particular $Y'$ embeds in $Y\times  [0,1 ]$.

\end{theorem}

\begin{remark}
 Yasui has pointed out to the author that, while they do not talk about cobordisms, this result, without the statement of only needing a single 1- and 2-handle pair, can be proven by putting together several results from \cite{Yasui13}.
\end{remark}

{\bf Acknowledgements.} The author is grateful to Chris Gerig, Robert Gompf, Kyle Hayden, Tye Lidman, Hyun Ki Min, Maggie Miller, Lisa Piccirillo, Danny Ruberman, and B\"ulent Tosun for various helpful discussions and comments on the first draft. The author would also like to thank Jen Hom and Tom Mark for help with aspects of Heegard Floer homology. The author would like to show his gratitude to Kouchi Yasui for mentioning the paper \cite{Yasui13} and having helpful discussions. The author is grateful to Marco Golla for asking the question which became Theorem~\ref{AQ2} and also for carefully reading the earlier drafts and making various helpful comments. And finally the author would like to show his gratitude to his advisor John Etnyre for constant help and support.

\section{Background}\label{background}
\subsection{Contact geometry} \label{contact}
  Recall that a \textit{(co-orientable) contact  structure} $\xi$ on an oriented 3-manifold $Y$ is the kernel of 1-form $\alpha \in \Omega ^1(Y)$ such that $\alpha \wedge d\alpha$ is non-degenerate. Geometrically a contact structure on a 3-manifold is a distribution of a 2-plane fields on the manifold that is not tangent to any embedded surface in the manifold. Darboux's theorem says that every contact 3-manifold $(Y,\xi)$ is locally contactomorphic to $(\R ^3,\xi_{std}= ker(dz - ydx))$. All orientable 3-manifolds admit contact structures. A knot $L \subset (Y,\xi)$ is called \textit{Legendrian} if at every point of $L$ the tangent line to $L$ lies in the contact plane at that point. A Legendrian knot $L$ in a contact manifold $(Y,\xi)$ has a standard neighborhood $N$ and a framing $fr_\xi$ given by the contact planes. If $L$ is null-homologous then $fr_\xi$ relative to the Seifert framing is the \textit{Thurston--Bennequin} invariant of $L$.  If one does $fr_\xi-1$ surgery on $L$ by removing $N$ and gluing back a solid torus so as to effect the desired surgery, then there is a unique way to extend $\xi|_{Y-N}$ over the surgery torus so that it is tight on the surgery torus. The resulting contact manifold is said to be obtained from $(Y,\xi)$ by \textit{Legendrian surgery} on $L$. 
  
  A Legendrian knot $L$ in $(\R^3, \xi_{std})$ projects to a closed curve $\gamma$ in the $xz$--plane which also known as \textit{front projection} of $L$. The curve $\gamma$ uniquely determines the Legendrian knot $L$ which can be reconstructed by setting $y(t)$ as the slope of $\gamma (t)$. Thus, at a crossing, the most negative slope curves always stays at front. There are two types of cusps singularity possible when $dz/dx =0$ whcih are called left cusps and right cusps. See Figure~\ref{front pr}. 
  
  Looking at an oriented front projection one can compute the Thurston--Bennequin invariant of a Legendrian knot $tb(K)= \text{writhe}(K)- \#\{\text{left  cusps}\}$. For more details we refer \cite{geiges} and \cite{gompf98}. 
  
  A contact 3-manifold $(Y,\xi)$ is called \textit{overtwisted} if there exists a Legendrian unknot with Thurston--Bennequin number $0$, otherwise it is called \textit{tight}.
  
  \begin{figure}[htbp]
	\begin{center}
	
  \begin{overpic}[scale=0.6,  tics=20]{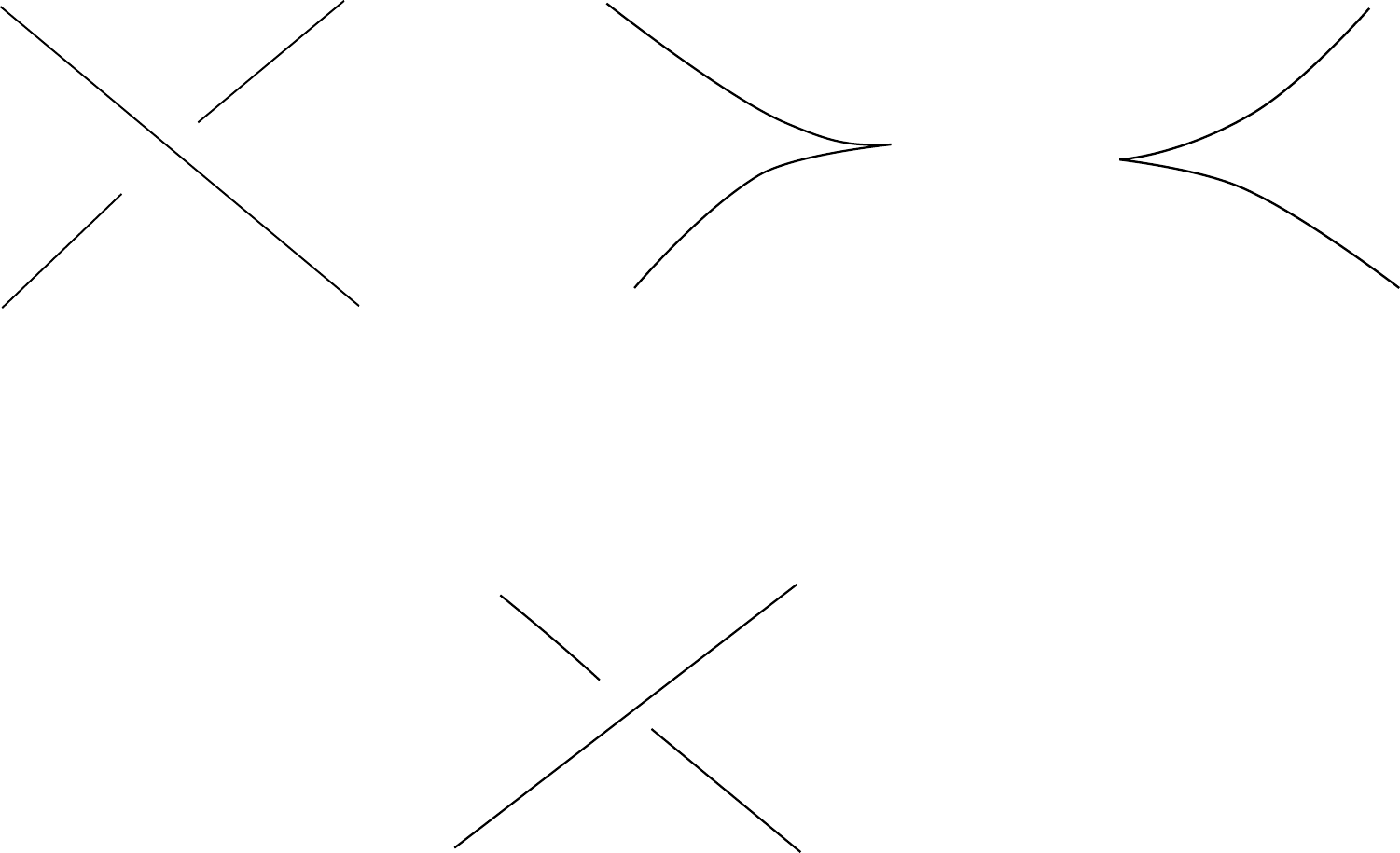}

\end{overpic}

\caption{The top row indicates the correct crossing and the cusps in the front projection. And the bottom picture crossing will not occur in the front projection diagram. }
\label{front pr}
\end{center}
\end{figure}

\subsection{Symplectic fillings, cobordisms and caps} \label{dstein}

We recall that a compact symplectic manifold $(X,\omega)$ is a \textit{strong  symplectic\  filling} of $(Y,\xi)$ if $\partial X=Y$ and there is a vector field $v$ defined near $\partial X$ such that the Lie derivative of $\omega$ satisfies $\mathcal{L}_v \omega= \omega$, $v$ points out of $X$ and $\iota_v\omega$ is a contact form for $\xi$. Moreover, $(X,\omega)$ is a strong symplectic cap for $(Y,\xi)$ if it satisfies all the properties above, except $\partial X=-Y$ and $v$ points into $X$. We also say $(X,\omega)$ is a \textit{weak  filling} of $(Y,\xi)$ if $\partial X=Y$ and $\omega|_\xi>0$  (here all our contact structures are co-oriented). Similarly, $(X,\omega)$ is a weak cap of $(Y,\xi)$ if $\partial X=-Y$ and $\omega|_\xi>0$. We shall say that $(Y,\xi)$ is (strongly or weakly) \textit{semi-fillable} if there is a connected (strong or weak) filling $(X,\omega)$ whose boundary is disjoint union of $(Y,\xi)$ with an arbitrary non-empty contact manifold.

A \textit{symplectic  cobordism} from the contact manifold $(Y_-,\xi_-)$ to $(Y_+,\xi_+)$ is a compact symplectic manifold $(W,\omega)$ with boundary $-Y_-\cup Y_+$ where $Y_-$ is a \textit{concave} boundary component and $Y_+$ is \textit{convex}, this means that there is a vector field $v$ near $\partial W$ which points transversally inwards at $Y_-$ and transversally outwards at $Y_+$, and $\mathcal{L}_v \omega= \omega$. The first result we will need concerns when symplectic cobordisms can be glued together. 
\begin{lemma}\label{glue} \cite{etnyre96}
Let $(X_i,\omega_i)$ be a symplectic cobordism from $(Y_i^-,\xi_i^-)$ to $(Y_i^+,\xi_i^+)$, for $i=1,2$, and $(Y_1^+,\xi_1^+)$ is contactomorphic to $(Y_2^-, \xi_2^-)$. Then we can construct a symplectic cobordism $(X,\omega)$ from $(Y_1^-,\xi_1^-)$ to $(Y_2^+,\xi_2^+)$ such that $X$ is diffeomorphic to $X_1\cup_{Y_1^+}X_2.$
\end{lemma}
\noindent

Now recall that a \textit{Stein  domain} is a triple $(X,J,\psi)$ where $J$ is a complex structure on $X$ and $\psi : X \to \R$ is a proper plurisubharmonic function, that is a smooth function such that $\psi^{-1}(-\infty,c]$ is compact for all $c\in \R$ and $\omega_{\psi}(v,w)= - d(d\psi \circ J)(v,w)$ is a symplectic form. A closed contact manifold $(Y,\xi)$ is called Stein fillable if there exists a  Stein manifold $(X,J,\psi)$ such that $\psi$ is bounded from below, $M$ is an inverse image of a regular value of $\psi$ and $\xi= ker(-d\psi \circ J)$. In fact we have the following charecterization of Stein 4-manifolds.

\begin{theorem}(Elaishberg \cite{yasha91}; Gompf \cite{gompf98}) \label{EG}
A 4-manifold is a Stein domain if and only if it has a handle decomposition with 0-handles, 1-handles, and 2-handles and the 2-handles are attached along Legendrian knots with framing one less than the contact framing.
\end{theorem}

Another way to build cobordisms is by Weinstein handle attachment, \cite{ Weinstein91}. One may attach a 0, 1, or 2-handle to the convex end of a symplectic cobordism to get a new symplectic cobordism with the new convex end described as follows. For a 0-handle attachment, one merely forms the disjoint union with a standard 4--ball and so the new convex boundary will be the old boundary disjoint union with the standard contact structure on $S^3$. For a 1-handle attachment, the convex boundary undergoes a, possibly internal, connected sum. A 2-handle is attached along a Legendrian knot $L$ with framing one less that the contact framing, and the convex boundary undergoes a Legendrian surgery. 

\begin{theorem}
Given a contact 3-manifold $(Y,\xi)$ let $W$ be a part of its symplectization, that is $(W= Y\times[0,1], \omega= d(\alpha e^t))$. Let $L$ be a Legendrian knot in $(Y,\xi)$ where we think of $Y$ as $Y\times \{ 1 \}$. If $W'$ is obtained from $W$ by attaching a 2-handle along $L$ with framing one less than the contact framing, then the upper boundary $(Y', \xi ')$ is still a convex boundary. Moreover, if the 2-handle is attached to a Stein filling (respectively strong, weak filling) of $(Y,\xi)$ then the resultant manifold would be a Stein filling (respectively strong, weak filling) of $(Y'\xi ')$.

\end{theorem}

The theorem for Stein fillings was proven by Eliashberg \cite{yasha91}, for strong fillings by Weinstein \cite{Weinstein91}, and was first stated for weak fillings by Etnyre and Honda \cite{eh02}.

Starting with a Stein filling (respectively strong, weak filling) of $(Y,\xi)$ one can construct a symplectic closed manifold by capping it off. Various people have studied concave caps on contact manifold but for our purpose we need the result of Etnyre, Min and the author \cite{mukherjee19}.

\begin{theorem}\label{cap}

If $(W,\omega)$ is weak filling of $(Y,\xi)$ then there exists a closed symplectic 4-manifold $(X,\omega')$ in which $(W,\omega)$ symplectically embeds such that the complement of $W$ in $X$ is simply-connected and has $b_2^+>0.$
\end{theorem}

\subsection{Heegard Floer homology}\label{4mfds} \label{heegard}

Recall that Heegaard Floer homology is an Abelian group associated to a 3-manifold $Y$, equipped with a $Spin^c$ structure  $\mathfrak{t}\in Spin^c(Y)$. These homology groups are invariant of the pair $(Y,\mathfrak{t})$ and are denoted by $HF^\infty(Y,\mathfrak{t})$, which is a $\Z [U,U^{-1}]$ module; $HF^+ (Y,\mathfrak{t})$, which is a $\Z[ U^{-1}]$ module; $HF^-  (Y,\mathfrak{t})$, which is a $\Z [U]$ module. These invariants fit into a long exact sequence
\[
\begin{tikzcd}
  \cdots \ar{r} & HF^-(Y,\mathfrak{t}) \ar{r} {\iota} & HF^\infty(Y,\mathfrak{t}) \ar{r} {\pi} & HF^+ (Y,\mathfrak{t}) \ar{r} {\delta} &  \cdots
\end{tikzcd}
 \]
 Recall that associated to this long exact sequence there is another 3-manifold invariant 
 \[
 HF^+_{red}(Y,\mathfrak{t})= Coker(\pi) \cong Ker(\iota) = HF^-_{red}(Y,\mathfrak{t})
 \]
 The isomorphism in the middle is induced by the co-boundary map. Recall that $d(Y,\mathfrak{t})$ is the minimum grading of the torsion-free elements in the image$\{\pi : HF^{\infty}(Y,\mathfrak{t})\rightarrow HF^+(Y,\mathfrak{t})\}$. For details, readers are referred to \cite{osd, os3}.
 
 Now recall that an \textit{$L$-space} $Y$ is a rational homology 3-sphere whose Heegard Floer homology is as simple as possible, that is  $HF^+_{red}(Y,\mathfrak{t})=0$ for all $Spin^c$ structures $\mathfrak{t}\in Spin^c(Y)$. 
 
 A cobordism between two 3-manifold induces a map on Heegad Floer homology. More precisely if $W$ is a cobordism from $Y_0$ to $Y_1$ and $\mathfrak{s}$ is a $Spin^c$ structure in $W$ whose restriction on $Y_i$ is denoted as $\mathfrak{s}_i$ for $i=0,1$, then there is a map $F^{\circ}_{W,\mathfrak{s}}: HF^\circ(Y_0,\mathfrak{s}_0)\rightarrow HF^\circ(Y_1,\mathfrak{s}_2)$, where $\circ = +,-\ or \ \infty $.

 
 \begin{theorem}(Ozsv\'{a}th-Szab\'{o} \cite{os4}) 
 If $W$ is a cobordism between $Y_0$ to $Y_1$ and $\mathfrak{s}$ is a $Spin^c$ structure on $W$ whose restriction on $Y_i$ is denoted as $\mathfrak{s}_i$ for $i=0,1$ then we have the following,

\[
\begin{tikzcd}
\cdots \ar{r} & HF^-(Y_0,\mathfrak{s}_0)\ar{d}{F^-_{W,\mathfrak{s}}} \ar{r}{\iota _0} & HF^\infty(Y_0,\mathfrak{s}_0) \ar{d}{F^{\infty}_{W,\mathfrak{s}}} \ar{r}{\pi_0} & HF^+ (Y_0,\mathfrak{s}_0) \ar{d} {F^+_{W,\mathfrak{s}}} \ar{r}{\delta _0} & \cdots \\
\cdots  \ar{r} & HF^-(Y_1,\mathfrak{s}_1) \ar{r}{\iota_1} & HF^\infty(Y_1,\mathfrak{s}_1) \ar{r}{\pi_1} & HF^+ (Y_1,\mathfrak{s}_1) \ar{r}{\delta_1} & \cdots
\end{tikzcd}
\]
where the vertical maps are uniquely determined up to an overall sign, and all the squares are commutative.
 
 \end{theorem}

 The \textit{composition law} states that if $W_0$ is a cobordism from $Y_0$ to $Y_1$ and $W_1$ is a cobordism from $Y_1$ to $Y_2$, and let $\mathfrak{s}_i$ be the $Spin^c$ structure on $W_i$ for $i=0,1$, then the relationship between composition of $F_{W_0,\mathfrak{s}_0}$ with $F_{W_1,\mathfrak{s}_1}$, and the maps induced by the composite cobordism $W=W_0\cup_{Y_1} W_1$ is
 \[
 F^\circ_{W_1,\mathfrak{s}_1} \circ F^\circ_{W_0,\mathfrak{s}_0}= \sum_{\{\mathfrak{s} \in Spin^c(W) |\  \mathfrak{s}|W_i=\mathfrak{s}_i,i=0,1\}} \pm F^{\circ}_{W,\mathfrak{s}}.
 \]
 
\subsection{Closed 4-manifold invariants}\label{cobordism}

 There is a variant of the cobordism invariant which is defined for cobordism with $b_2^+(W)>1$. This following lemma is proven by Ozsv\'{a}th and Szab\'{o} \cite{os4}
 
 \begin{lemma}
 Let $W$ be a cobordism between $Y_0$ and $Y_1$ with $b_2^+(W)>0$. Then the induced cobordism map $F^\infty_{W,\mathfrak{s}}$ vanishes for all $Spin^c$ structures on $W$. 
 \end{lemma}
 
 If we have a cobordism $W$ with $b_2^+(W)>1$, then we can cut $W$ along a 3--manifold $N$, which divides $W$ into two cobordism $W_0$ and $W_1$, both of which have $b_2^+(W_i)>0$, in such a way that the map induced by the restriction
 \[
 Spin^c(W)\rightarrow Spin^c(W_0)\times Spin^c(W_1)
 \]
 is injective. Such a cut $N$ is called $admissible \  cut$. 
 \begin{remark}
 Notice that if in a cobordism $W$ with $b_2^+(W)>1$ we find a separating  rational homology 3--sphere $N$ such that both the pieces have $b_2^+>0$, then $N$ is an admissible cut.
 \end{remark}
 
 If $\mathfrak{s}$ is a $Spin^c$ structure on $W$ whose restriction to $W_i$ is $\mathfrak{s_i}$ and the induced $Spin^c$ structures on 3-manifolds $Y_0,Y_1$ and $N$ is $\mathfrak{t}_0,\mathfrak{t}_1$ and $\mathfrak{t}$ , then 
 \[
  F^-_{W_0,\mathfrak{s}_0}: HF^-(Y_0,\mathfrak{t}_0)\rightarrow HF^-(N,\mathfrak{t})
  \]
 factors through the inclusion $HF^-_{red}(N,\mathfrak{t})\rightarrow HF^-(N,\mathfrak{t})$, and
 \[
  F^+_{W_1,\mathfrak{s}_1}: HF^+(N,\mathfrak{t})\rightarrow HF^+(Y_1,\mathfrak{t}_1)
 \]
 factors through the projection $HF^+(N,\mathfrak{t})\rightarrow HF^+_{red}(N,\mathfrak{t})$. And thus by using the identification of $HF^+_{red}(N,\mathfrak{t})\cong HF^-_{red}(N,\mathfrak{t})$ in the middle, we can define the \textit{mixed  invariant} as a map 
 \[
 F^{mix}_{W,\mathfrak{s}}: HF^-(Y_0,\mathfrak{t}_0) \rightarrow HF^+(Y_1,\mathfrak{t}_1).
 \]

\begin{remark}
 
 It is also proven in \cite{os4} that $F^{mix}$ does not depend on the choice of the admissible cut.
 \end{remark}
 
  From the discussion above one immediately sees the following result.
\begin{lemma}

 If an admissible cut $N$ of $W$ is an $L$-space then $F^{mix}_{W,\mathfrak{s}}$ vanishes.
\end{lemma}


\begin{theorem}(Ozsv\'{a}th, Szab\'{o} \cite{oss})
If $(X,\omega)$ is a closed, symplectic manifold with $b_2^+(X)>1$, then for the canonical $Spin^c$ structure $\mathfrak{k}$ corresponding to the symplectic form, $F^{mix}_{X,\mathfrak{k}}$ is non-vanishing. Here we think of $X$ as a cobordism from $S^3$ to $S^3$ by taking out two 4-balls.

\end{theorem}


 \begin{remark}\label{remark}
The above discussion implies that an $L$-space cannot be an admissible cut for a closed symplectic 4-manifold with $b_2^+>1$.
\end{remark}
 
\section{Topological embedding of 3-manifolds in symplectic 4-manifolds}\label{caps}

Now we will begin the proof of topological embedding of 3-maifolds into symplectic 4-manifolds.

\begin{proof}[Proof of Theorem~\ref{ttop}]
We will topologically embed a 3-manifold $Y$ into a symplectic manifold $X$ in three steps. In the fourth step we will show that the embedding is smooth after a single stabilization with $S^2\times S^2$. We start with a Kirby picture, consisting of only a $0$-handle and 2-handles, for a 4--manifold whose boundary is $Y$.

\begin{step}
 {Stein modification of the Kirby picture.}
\end{step}
Let $K_1,\ldots, K_m$ be the attaching spheres for the 2-handles. We can Legendrian realize the $K_i$ so that each $K_i$ intersects a fixed Darboux ball $B$ in a horizontal arc. We can blow up meridians to each $K_i$ so that the framing on $K_i$ is less than $tb(K_i)-2$. All the blown-up unknots can be gathered in the Darboux ball as shown in the top left of Figure~\ref{convert stein}. 
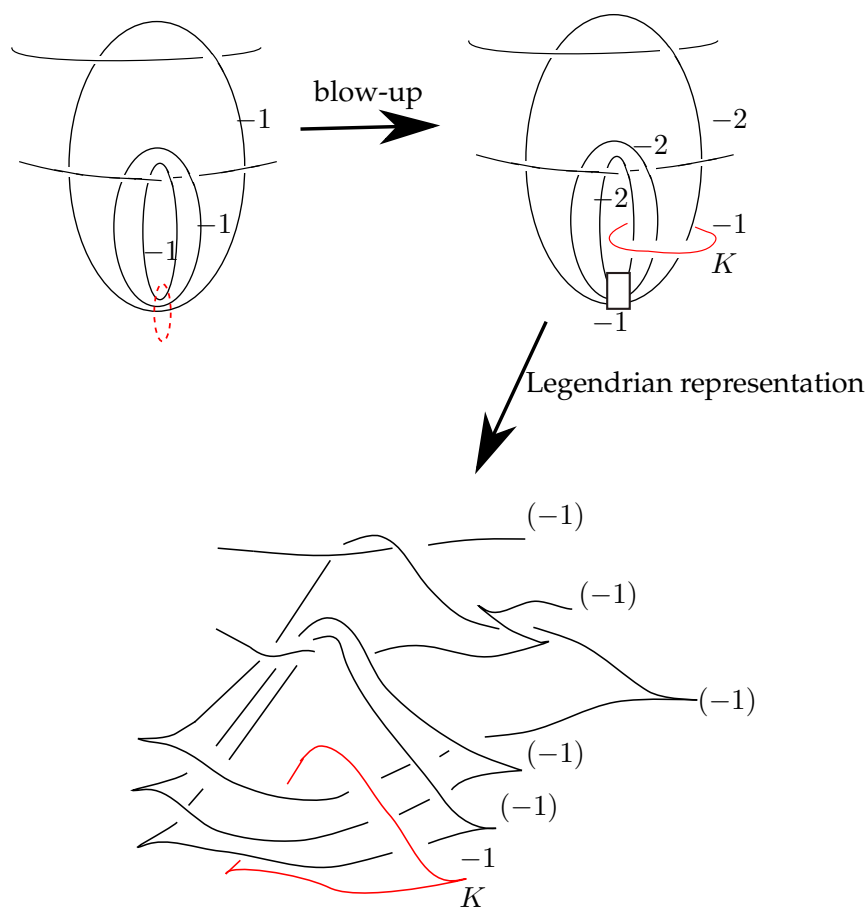
\begin{figure}[htbp]
	\begin{center}
  \begin{overpic}[scale=0.6,  tics=20]{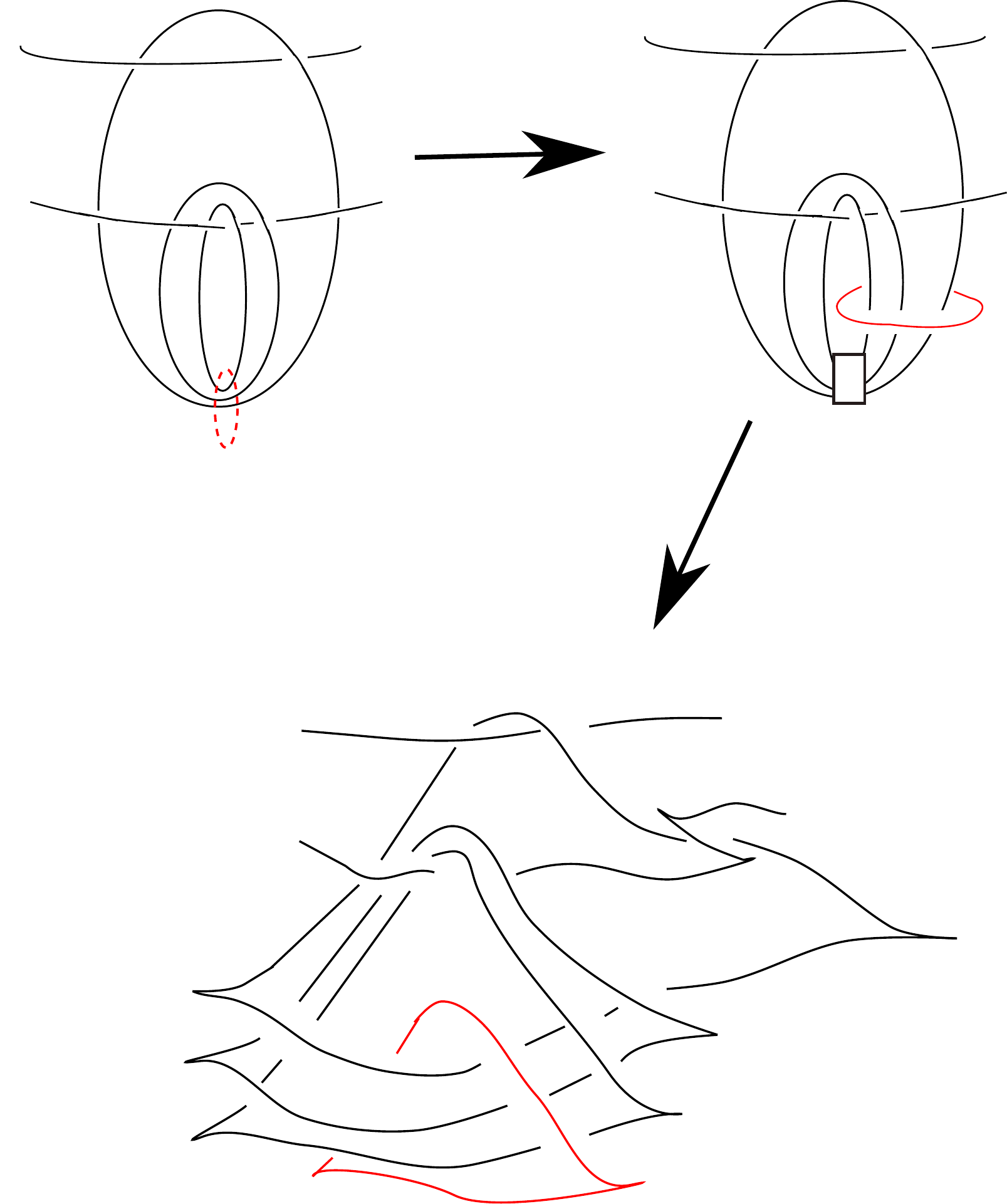}
    \put(90,290){$-1$}
    \put(75,250){$-1$}
    \put(55,240){$-1$}
    \put(60,200){}
    \put(120,300){blow-up}
    \put(270,290){$-2$}
    \put(240,280){$-2$}
    \put(225,260){$-2$}
    \put(270,250){$-1$}
    \put(270,235){$K$}
    \put(225,213){$-1$}
    \put(200,190){Legendrian representation}
    \put(200,140){$(-1)$}
    \put(265,70){$(-1)$}
    \put(220,110){$(-1)$}
    \put(200,50){$(-1)$}
    \put(190,30){$(-1)$}
    \put(175,10){$-1$}
    \put(175,-5){$K$}
\end{overpic}
\caption{Converting a Kirby picture of the 3-manifold by blowing up such that the complement of the red knot is Stein. Here $(-1)$ is measured relative to the contact framing.}
\label{convert stein}
\end{center}
\end{figure}
Blow up one more unknot as indicated in the upper right of Figure~\ref{convert stein} and notice that the resulting link $L$ can be Legendrian realized as in the bottom diagram of Figure~\ref{convert stein}. Let $K$ be the unknot with framing $-1$ in the figure. We now stabilize the components of $L-K$ so that the Thurson-Bennequin invariant of each component is one larger than its surgery coefficients. Legendrian surgery on $L-K$ together with $-1$ surgery on $K$ gives our manifold $Y$. (Notice that to realize $L$ each $K_i$ might need to be stabilized one extra time as shown in the figure. This is why we arranged the surgery coefficients to be less than $tb(K_i)-2$.) So the manifold $W_0$ obtained by attaching Stein 2-handles to $L-K$ is a Stein manifold and we denote the boundary by $Y_0$. Now attaching a 2-handle to $W_0$ along $K$ with framing $-1$ gives a 4-manifold $W$ with boundary $Y$. 


\begin{step}\label{step2}
Attach a cork and apply cork-twist.
\end{step}
We begin by constructing a manifold $W_2$ by modifying the surgery presentation that is add the 1- and 2-handle shown in Figure~\ref{stein}, where the 2-handle links $K$ as indicated and is otherwise disjoint from $L$ ( by abusing of language we will call this operation as attaching Mazur cork). Said another way, we can build a cobordism $W_1$ by attaching the 1- and 2-handle to $Y\times [0,1]$ along $Y\times \{1\}$. The manifold $W_2$ is now simply $W\cup W_1$ with $\partial W$ glued to $-Y\subset W_1$ and $W_1$ is a cobordism from $Y$ to some manifold $Y'$.

We can apply a cork-twist by interchanging the 1-handle and the $0$-framed 2-handle.
The cork twist does not change the boundary 3-manifold $Y'$. After the cork twist the knot $K$ is passing over the 1-handle geometrically once and thus they cancel each others. After this handle cancellation the knot $K$ in the original picture Figure~\ref{convert stein} is replaced by $-1$ framed knot in the third picture of Figure~\ref{stein}. Notice that this new knot can be realized by a Legendrian knot that has Thurston-Bennequin invariant $+1$ and thus $-1$ smooth surgery on this knot can be realized as Legendrian surgery on a stabilization of the knot. So we get a Stein filling of the boundary.

 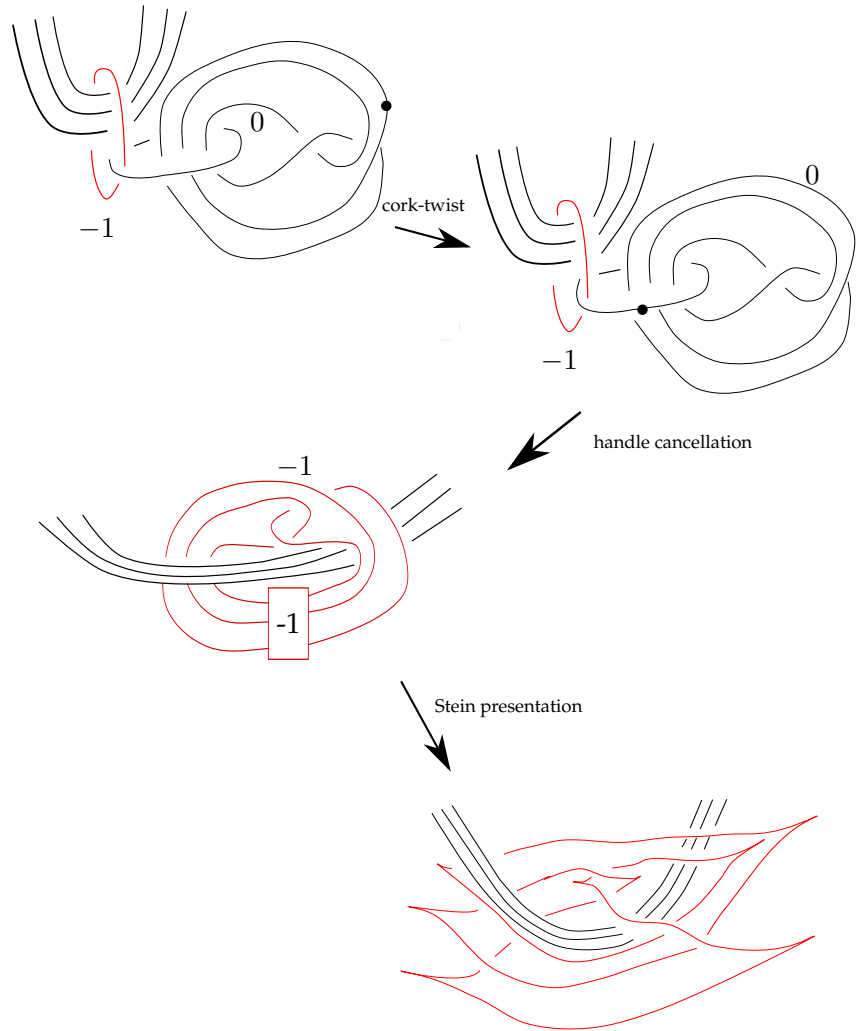
\begin{figure}[htbp]
	\begin{center}
  \begin{overpic}[scale=0.6,  tics=20]{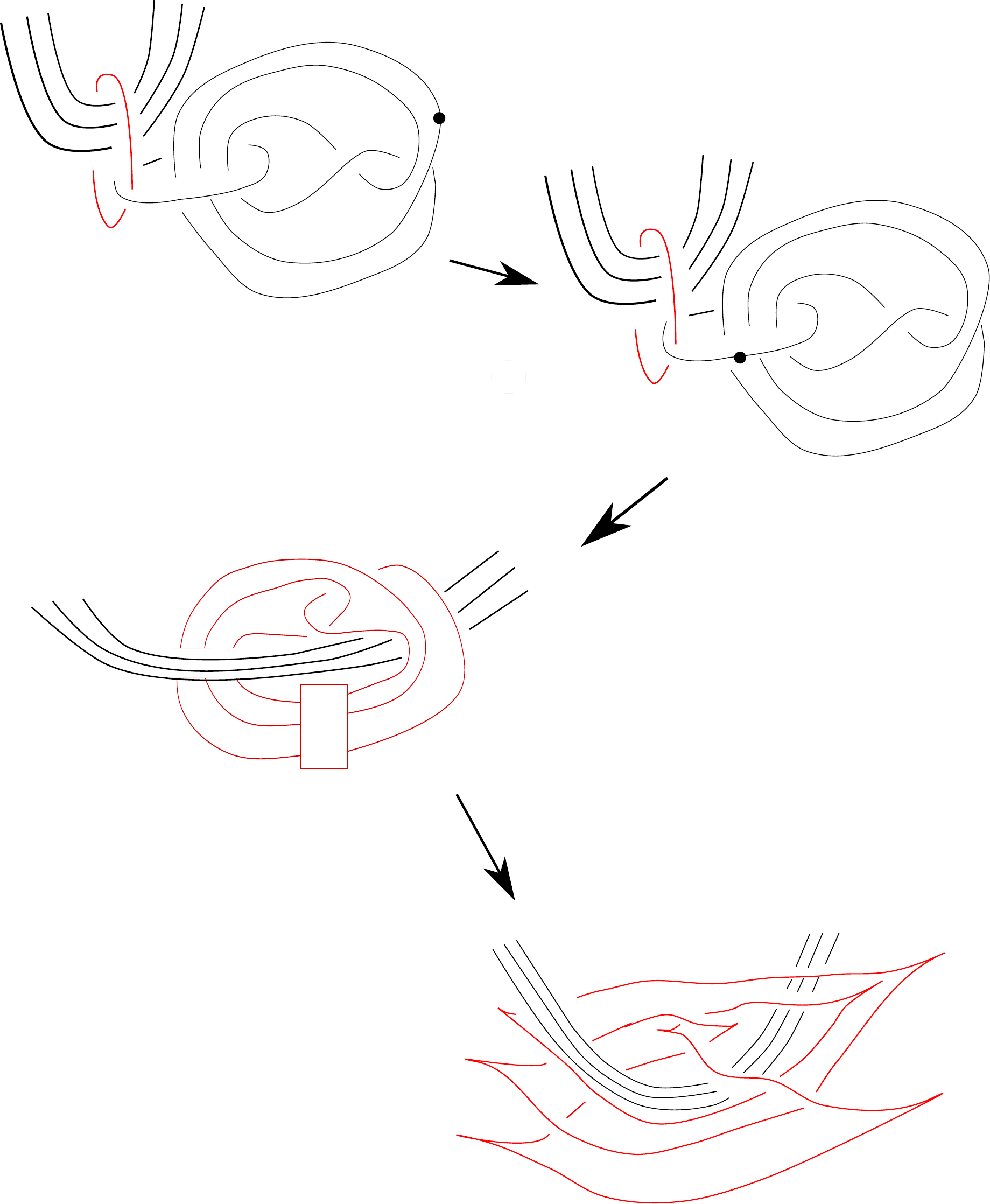}
\put(25,300){$-1$}
\put(90,340){$0$}
\put(200,250){$-1$}
\put(140,310){\tiny{cork-twist}}
\put(300,320){$0$}
\put(220,220){\tiny{handle cancellation}}
\put(100,210){$-1$}
\put(100,150){-1}
\put(160,120){\tiny{Stein presentation}}
\end{overpic}

\caption{
The maximal Thurston-Benequin number of the red knot in the bottom picture is $+1$. So it is a Stein $2-$handle attachment.}
\label{stein}
\end{center}
\end{figure}

\begin{step} Construct a simply connected a closed symplectic 4-manifold. 
\end{step}

 Also notice that the $W_1$ deformation retract onto $Y$ and $W$ is a 2-handlebody, so in particular $W'_2$ is simply connected. So in Step~\ref{step2} when we do the cork-twist the new manifold $W_2'$ is still homeomorphic to the compact manifold $W_2$ by a result of Freedman \cite{freedman} and thus the original 3-manifold has a topological, locally flat embedding in $W_2'$. Now we use the simply connected cap constructed in Theorem~\ref{cap} to cap off the upper boundary $W'_2$ to get a closed simply connected symplectic 4-manifold $X$ into which the 3-manifold $Y$ topologically, locally flatly embeds.

\begin{step}
Smooth embedding after one stabilization.
\end{step}

We can stabilize $X$ by adding a Hopf link to $W_2'$ and using the same cap (of course this stabilized 4-manifold is no longer symplectic). We can now handle slide one of the components $C$ of the Hopf link over the $0$-framed knot in the Mazur cork as indicated at the top of Figure~\ref{smooth}. Using the $0$-framed meridian to $C$ we can untangle $C$ from the 2-handle in the Mazur cork as shown in the middle of Figure~\ref{smooth}. We can further slide $C$ over the $0$-framed meridian to turn $C$ into a meridian of the 1-handle. Thus the 1-handle can be cancelled with $C$, leaving the bottom picture in Figure~\ref{smooth}. Thus we see a smooth embedding of $Y$ into $W_2'\# S^2\times S^2$ and thus into $X \# S^2\times S^2$.
\end{proof}

 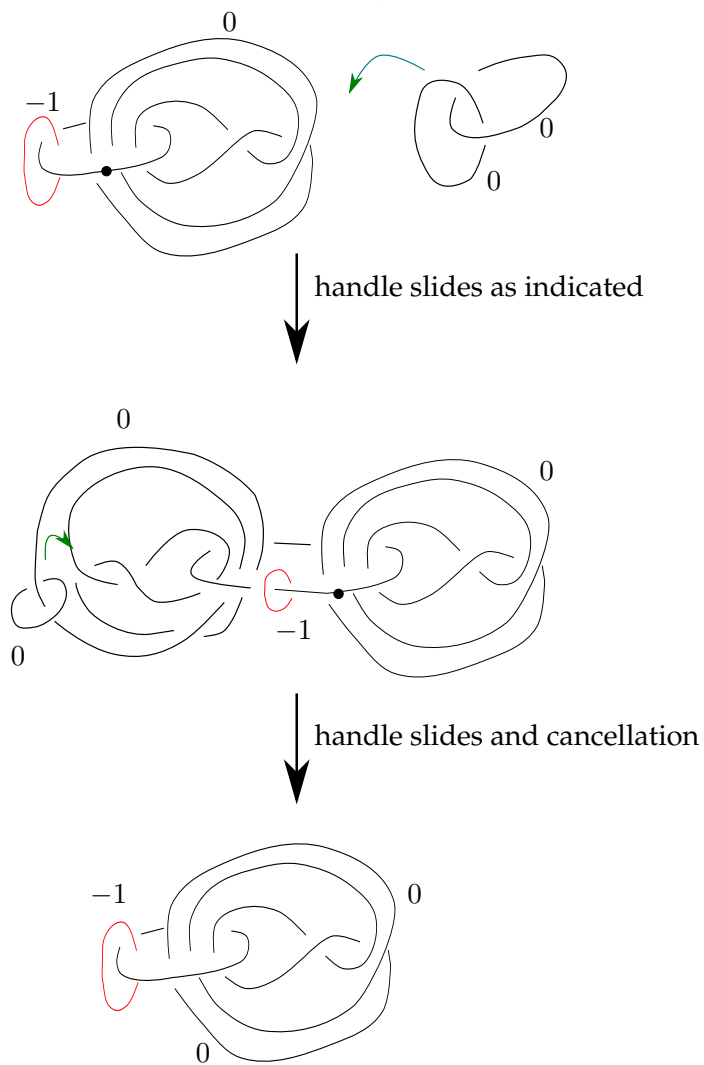
\begin{figure}[htbp]
	\begin{center}
  \begin{overpic}[scale=0.6, tics=20]{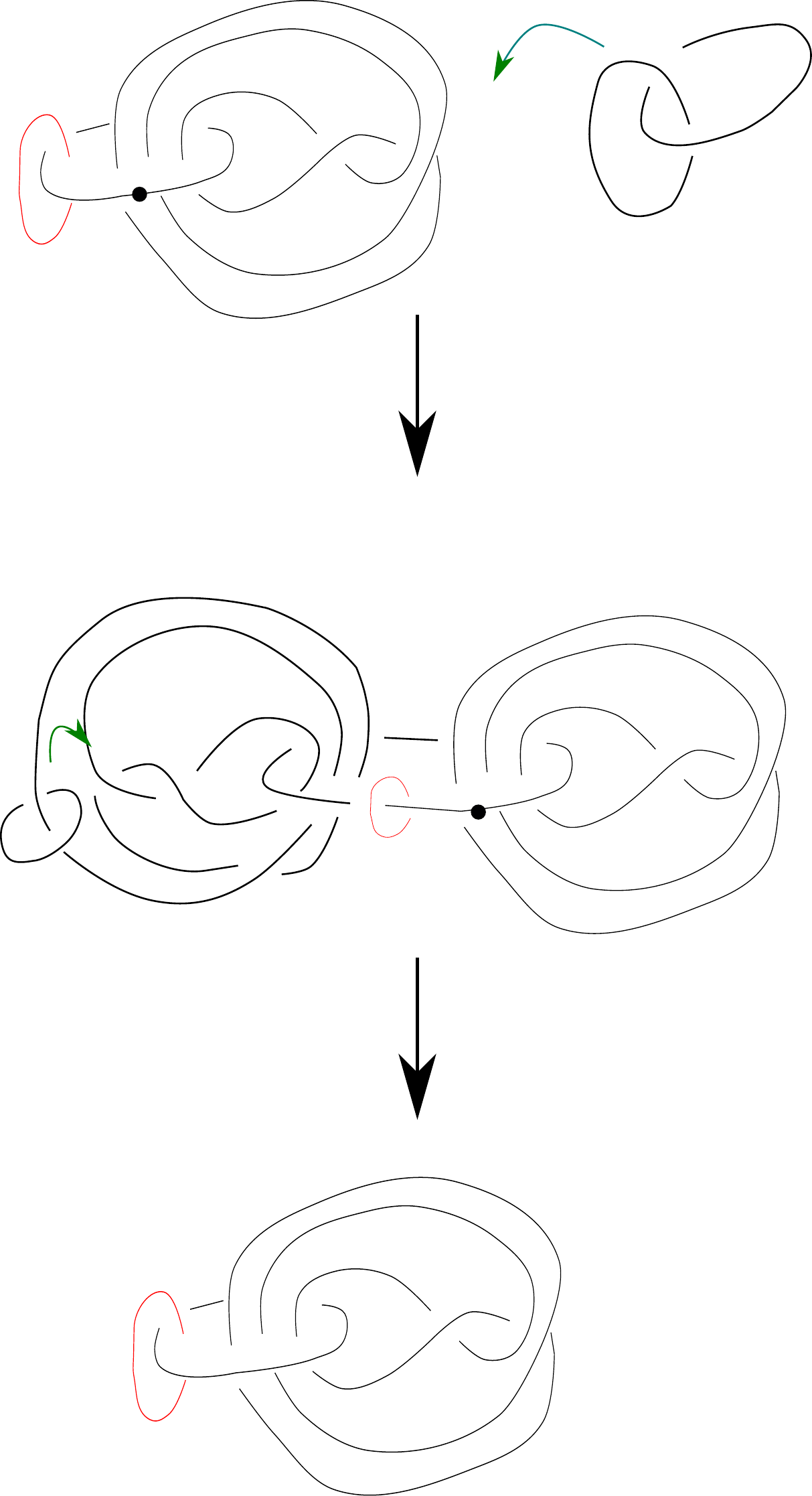}
  \put(5,360){$-1$}
  \put(80,390){$0$}
  \put(180,330){$0$}
  \put(200,350){$0$}
  \put(115,290){handle slides as indicated}
  \put(0,150){$0$}
  \put(40,240){$0$}
  \put(100,160){$-1$}
  \put(200,220){$0$}
  \put(115,120){handle slides and cancellation}
  \put(30,60){$-1$}
  \put(70,0){$0$}
  \put(150,60){$0$}

\end{overpic}

\caption{In this picture we are describing the Kirby moves of how connected summing with $S^2 \times S^2$ helps to cancel the 1-handle of the cork.}

\label{smooth}
\end{center}
\end{figure}

We now turn to Theroem~\ref{c1} that says given any 3-manifold $Y$ there is a simple invertible $\Z$ ribbon homology cobordism to a Stein fillable manifold $Y'$.

\begin{proof}[Proof of Theorem~\ref{c1}]

The cobordism $W_1$ from Step~\ref{step2} is the desired ribbon $\Z$-homology cobordism which is attaching a cork along the red knot on top Picture~\ref{stein}.

Let $D(W_1)$ be the double of $W_1$ along $Y'$, that is glue an upside down copy of $W_1$ on top of $W_1$ along the boundary. If $h$ is the 2-handle in $W_1$, then $D(W_1)$ is formed by attaching a 2- and 3-handle to $W_1$, with the 2-handle attached to a $0$-framed meridian to $h$ \cite[Section 4.2]{gs}. Check Figure~\ref{trivial}. Now by changing crossings on the attaching circle of $h$ using the $0$-framed meridian we can arrange that $h$ is passing over the 1-handle geometrically once. Thus they cancel each other. And after the cancellation the $0$-framed meridian $h$ will cancel the 3-handle. And thus the resultant manifold $D(W_1)= Y\times I$.
 \begin{figure}[htbp]
 
	\begin{center}
  \begin{overpic}[scale=0.6,   tics=20]{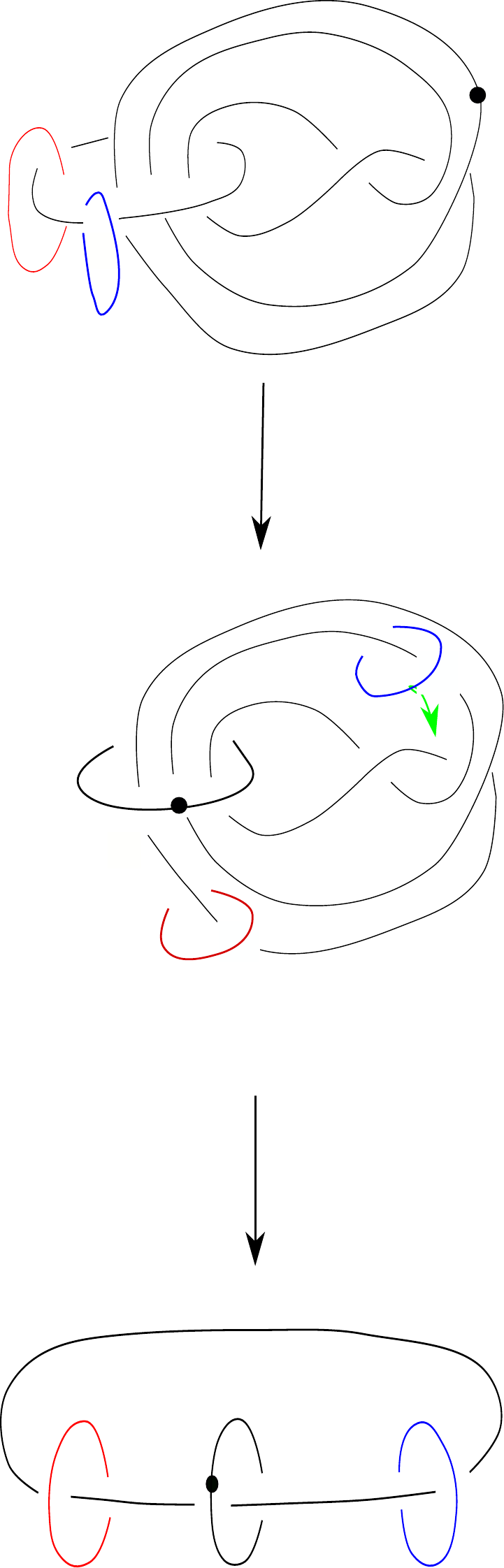}
 
  \put(-60,340){$D(W_1) \ =$}
  \put(140,340){$\cup \ 3-handle$}
  \put(75,275){ Isotopy}
  \put(75,100){Handle-slides}
  \put(20,300) {0}
  \put(80,220){0}
  \put(130,220){0}
  \put(65,340){0}
  \put(140,200){$\cup \ 3-handle$}
  \put(120,60){0}
  \put(120,0){0}
  \put(140,40){$\cup \ 3-handle$}
 
 \end{overpic}

\caption{
After isotopy, in the second picture the 0-framed blue 2-handle will help to resolve the crossings of the black 2-handle so that it can cancel the 1-handle. And after that the 3-handle will cancel the 0-framed unknotted blue 2-handle.}
\label{trivial}
\end{center}
\end{figure}
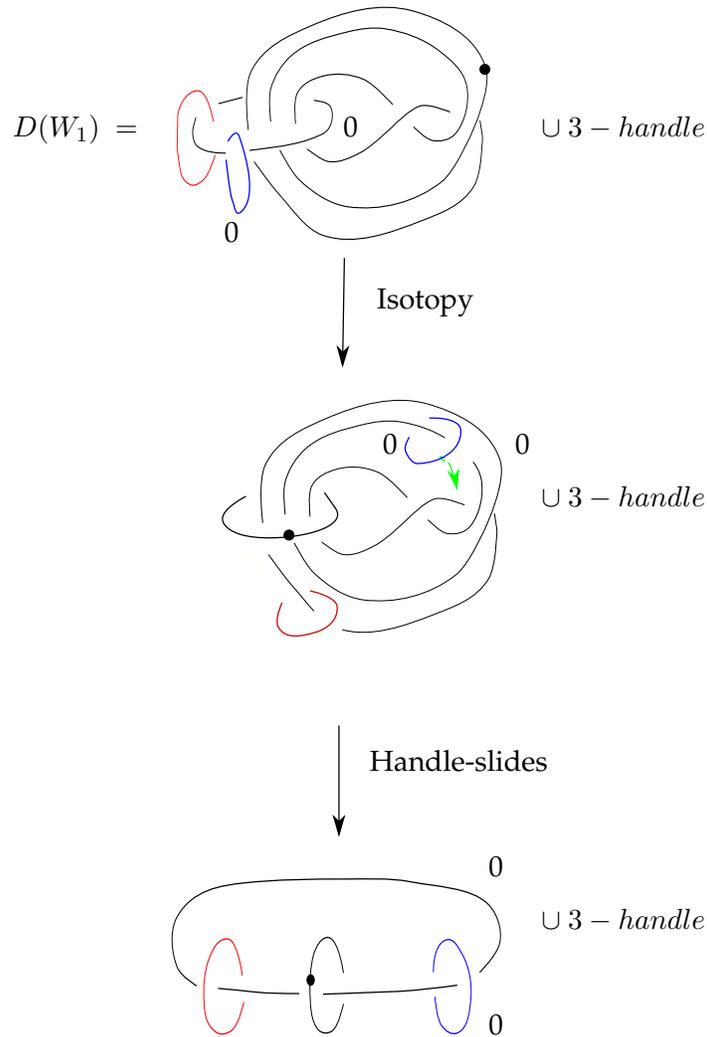
\end{proof}

 We now prove that 3-dimensional homology cobordism group is generated by Stein fillable manifolds
\begin{proof}[Proof of Theorem~\ref{cemb}]
Theorem~\ref{c1} provides a homology cobordism from any manifold to a Stein manifold. Thus the homology cobordism groups are generated by Stein manifolds.
\end{proof}

 we now prove the existence of degree 1 map.

\begin{proof}[Proof of Theorem~\ref{deg}]
Restriction to $Y'$ of the projection map on $Y\times I$ onto $Y$ will induces a degree 1 map from $Y'\rightarrow Y$.
\end{proof}

\section{Embedding L--spaces in symplectic 4-manifolds}\label{cobord}
We now give restrictions on embeddings of $L$-spaces in symplectic manifolds.

\begin{proof} [Proof of Theorem~\ref{psep}]
Suppose $Y$ is an $L$-space that smoothly embeds in a closed 4-manifold. We begin by showing it is separating. To this end we assume it is non-separating. Let $X_1$ be the compact manifold obtained from $X$ by cutting along $Y$. Notice that $\partial X_1= Y \sqcup -Y$, so we can glue two copies $X_1^1,X_1^2$ of $X_1$ along their boundaries to get a closed manifold $X'$. As constructed, $X'$ is a double cover of $X$ so, in particular, we can lift the symplectic form using the covering map and thus $X'$ is symplectic. Let $N$ be a neighbourhood of an arc in $X_1^1$ connecting its boundary components. Set $X_1' = \overline{X_1^1-N}$ and $X_2'=X_1^2\cup N$. Clearly $\partial X_i'= Y \# -Y$ which is an $L$-space. As $X$ is symplectic and $Y$ is a rational homology sphere, by using the Mayer-Vietoris sequence we can see that $b_2^+(X_i')= b_2^+(X)>0$. So $Y\# \overline{Y}$ is an admissible cut for a symplectic manifold $X'$ which contradicts Remark~\ref{remark}.   

Now when $Y$ embeds in $X$ in a separating manner then one of the component of $X-Y$ must have $b_2^+=0$ or we will get the same contradiction as before.
\end{proof}



\begin{remark}\label{strategy}
 
We now discuss a strategy to show negative answer to Question~\ref{1} about $L$-spaces bounding definite $4$-manifolds. 
Before that notice, all lens spaces bound both positive-definite and negative-definite 4-manifolds, because every lens space can be thought of as the boundary of a negative plumbed manifold and $-L(p,q)=L(p,p-q)$. One can obstruct rational homology spheres $Y$ with $H_1(Y,\Z)\neq 0$ from bounding negative-definite manifolds by using the technique developed by Owens and Strle \cite{owens12} where in Theorem 2 they proved that if maximum value of $d$-invariant of $Y$ is smaller than $1/4$ (with some more algebraic conditions) then $Y$ cannot bounds a negative-definite 4 manifold. Now $d(Y,\mathfrak{s})=-d(-Y,\mathfrak{s})$, so if we find an $L$-space $Y$ with $H_1(Y,\Z)\neq 0$, for which the absolute differences between $d$-invariants for different $Spin^c$ structures are very small then that could be used to obstruct it from bounding positive-definite and negative-definite 4-manifolds (as both $Y$ and $-Y$ cannot bound negative-definite 4-manifolds). So we can ask,

\end{remark}

\begin{question} For every $n\in \N$ does there exist an $L$-space which is not an $\Z HS^3$ and whose $d$-invariant values are in $(-1/n, 1/n)$?

\end{question}

We now prove Theorem~\ref{lemb} that says an $L$-space that does not bound negative-definite 4-manifold cannot have an oriented cobordism embedding in a compact symplectic 4-manifold with convex boundary.
 
\begin{proof}[Proof of Theorem~\ref{lemb}]
Let $Y$ be an $L$-space that does not bound negative definite 4-manifold.  If $Y$ embed in any symplectic 4-manifold with weakly convex boundary $W$ then it has to be separating since otherwise we can cap off with a concave cap to get a closed symplectic manifold where $Y$ is non-separating which contradicts Proposition~\ref{psep}. So $Y$ has to be separating. When  we cap off the upper boundary of $W$ by a cap with $b_2^+ >0$, since $Y$ does not bound a negative definite 4-manifold, both the sides of $Y$ have $b_2^+>0$. In particular $Y$ is an admissible cut for a symplectic 4 manfiold with $b_2^+>1$ which is a contradiction by Remark~\ref{remark}.
\end{proof}

\begin{proof}[Proof of Corollary~\ref{YxI}]
Let $Y'$ admit weakly fillable contact structure and $Y$ be an $L$-space that does not bound a negative definite 4-manifold. If $Y$ has an oriented  cobordism embedding in $Y'\times [0,1]$, then since $Y'$ is weakly fillable Y has an oriented  cobordism embedding in a symplectic 4 manfiold with weakly convex boundary, contradicting Theorem~\ref{lemb}.
\end{proof}

We now show the existence of exotic manifolds with boundary and $b_2=1$ using the ideas above.

\begin{proof}[Proof of Corollary~\ref{cex}]
Now start with $B^4$ and attach a 2-handle $h$ along  pretzel knot $P(-2,3,7)$ to get a 4-manifold $W'$ with $S^3_9(K)$  as its boundary. Attach a cork as step~\ref{step2} of the Theorem~\ref{ttop} and get $W$ with $b_2^+(W)=1$. After a cork-twist we can see that the 2-handle $h$ now passing over the 1-handle of the cork and this will increase the contact framing of $h$ by one as in Figure~\ref{tb} and thus the resulting manifold $W'$ will be Stein by Theorem~\ref{EG}. Before the cork-twist we had a smooth embedding of $S^3_9(K)$ in $W$. But by Theorem~\ref{lemb} $S^3_9(K)$cannot embed smoothly in $W'$ so they are exotic pairs.
\end{proof}

\section{Constructing rational ribbon cobordism}

\begin{proof}[Proof of the Theorem~\ref{cstein}]
Let $X$ be a compact oriented 4-manifold with boundary $Y$ and $b_1(X)=0$. Turning a handle structure on $X$ upside down, we can think of $X$ as a cobordism from $-Y$ to $\emptyset$, this is indicated in Figure~\ref{cob}.
\begin{figure}[htbp]
	\begin{center}
  \begin{overpic}[scale=0.6,  tics=20]{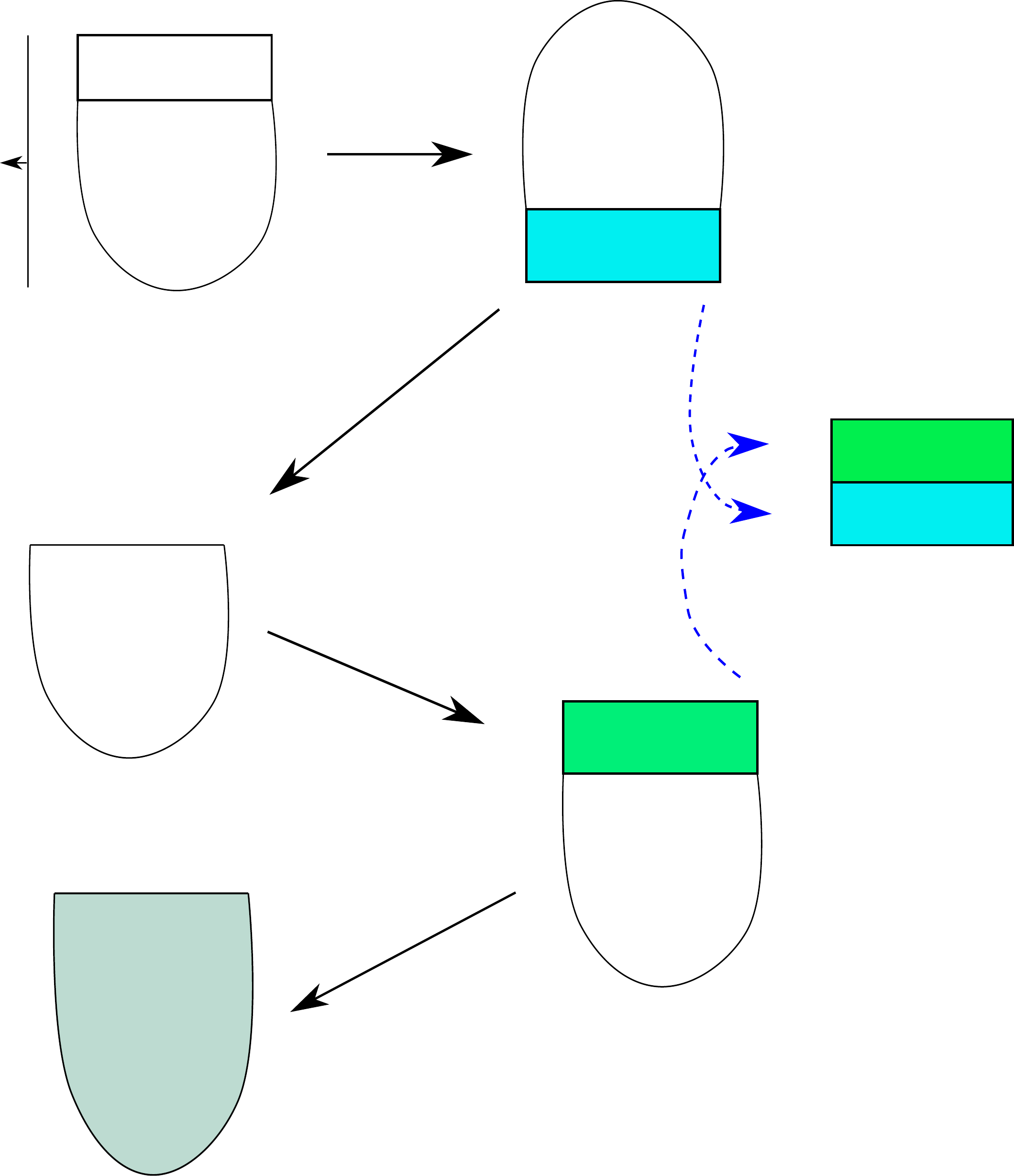}
  \put(-15,365){$X$}
 \put(35,400){3-handles}
 \put(100,420){Y}
 \put(110,380){\tiny{ turn upside-down}}
 \put(220,330){$M_2$}
 \put(260,315){-Y}
  \put(265,350){$Y_1$}
  \put(125,260){\tiny{slice off $M_2$ and turn upside-down}}
  \put(35,190){$X_1$}
  \put(100,175){\tiny{add corks}}
  \put(80,230){$Y_1$}
  \put(230,155){$M_3$}
  \put(280,170){$Y'$}
  \put(140,100){\tiny{cork-twist}}
  \put(45,50){$X'$}
  \put(325,210){$-Y$}
  \put(90,100){$Y'$}
  \put(325,280){$Y'$}
  \put(325,235){$M_2$}
  \put(325,255){$M_3$}

\end{overpic}

\caption{
A schematic of the construction of a ribbon cobordism from $Y$ to a Stein fillable $Y'$.}
\label{cob}
\end{center}
\end{figure}
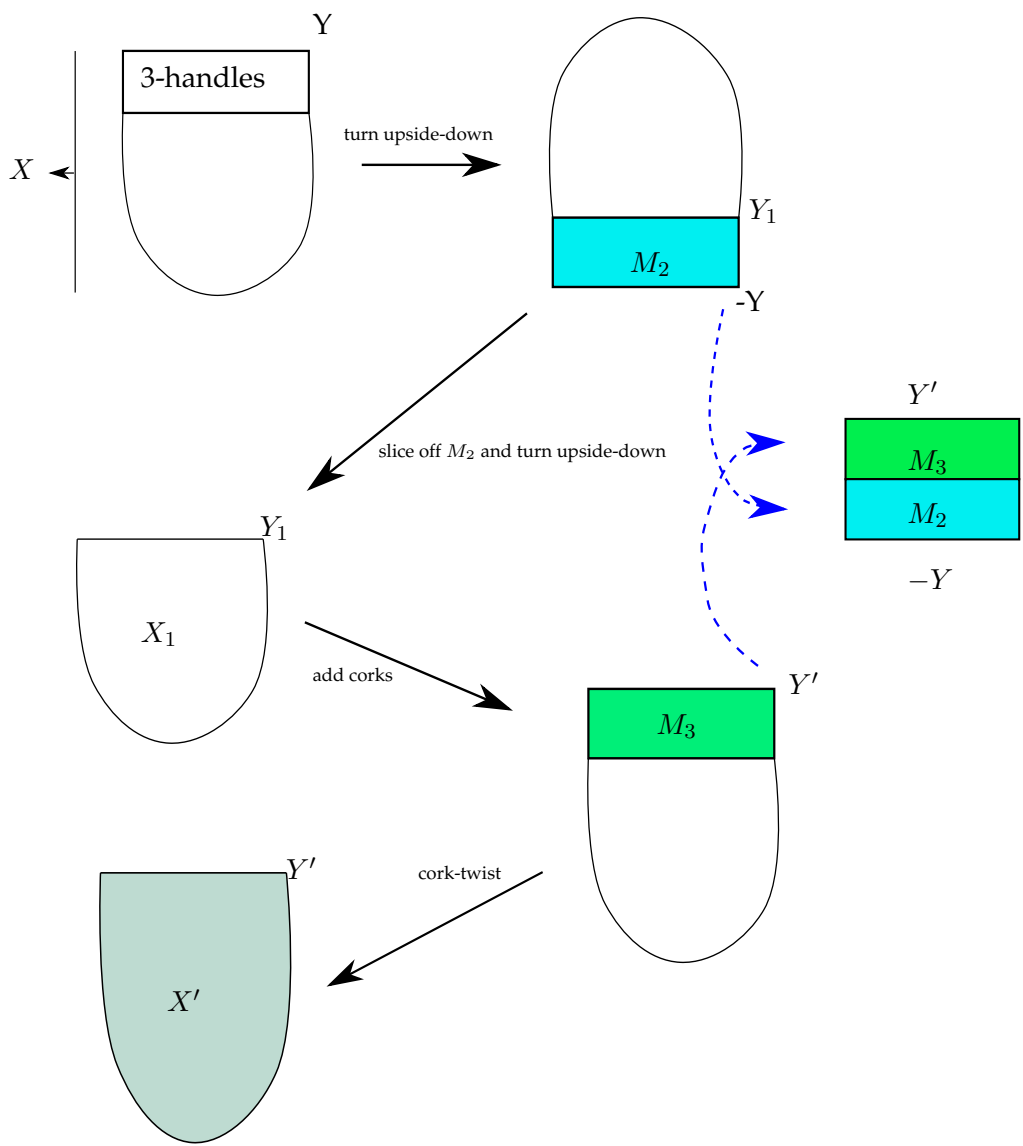
Notice that in this upside-down cobordism all the 1-handles of $X$ are converted into 3-handles and all the 3-handles become 1-handles. In the upside-down $X$, 1-handles are attached onto $-Y\times [0,1]$ along $-Y\times \{ 1\}$, let us call this cobordism $M_1$. Notice that $b_1(X)=0$ so the homology long exact sequence of the pair $(X,Y)$ implies that  there exists a minimal set of 2-handles such that if we attach those on top of $M_1$, and let us call it $M_2$, then $H_1(M_2,Y;\Q))=0$. (Here by minimum we mean that if we take any 2-handle out from the set then $H_1(M_2,Y;\Q)\neq 0$.) Since we consider a minimal set of 2-handles for this construction, we have $H_2(M_2,Y;\Q)=0$ as well because in this case the number of 1-handles of $M_2$ is the same as the number of 2-handles. Thus $M_2$ is a rational ribbon cobordism from $Y$ to say $Y_1$ which is the top boundary of $M_2$, see the top right of Figure~\ref{cob}. Consider $X_1$ to be the  handlebody obtained from $X$ by taking out $M_2$ and turning what remains upside-down, this is indicated in the third picture Figure~\ref{cob}. Thus $X_1$ only has 1- and 2-handles with boundary $Y_1$. If this is Stein then we are done. If not then that implies it has some 2-handles whose smooth framing is bigger than that the contact framing minus 1 of the attaching circle in $\# S^1\times S^2$ (in this case we can think of the top boundary $Y_1$ is obtained after attaching 2-handles on the boundary of 1-handlebody which is connected sum of $S^1\times S^2)$. To fix this framing issue, we repeatedly apply the Step~\ref{step2} of the proof of Theorem~\ref{ttop}. That is we attach a cork as in Figure~\ref{stein} (where the red curve there is the handle that needs its Thurston--Bennequin invariant increased). We then do a cork twist that exchange the 1-and 2-handles. We claim this has the effect of increasing the contact framing of the original attaching sphere of the 2-handle by 1. To see this, notice that if a knot passing over 1-handle then in the front projection diagram of a knot we are actually deleting two  consecutive right and left cusps by connecting them through a 1-handle and thus we are increasing the contact framing. See Figure~\ref{tb}. 
\begin{figure}[htbp]
	\begin{center}
  \begin{overpic}[scale=0.6,  tics=20]{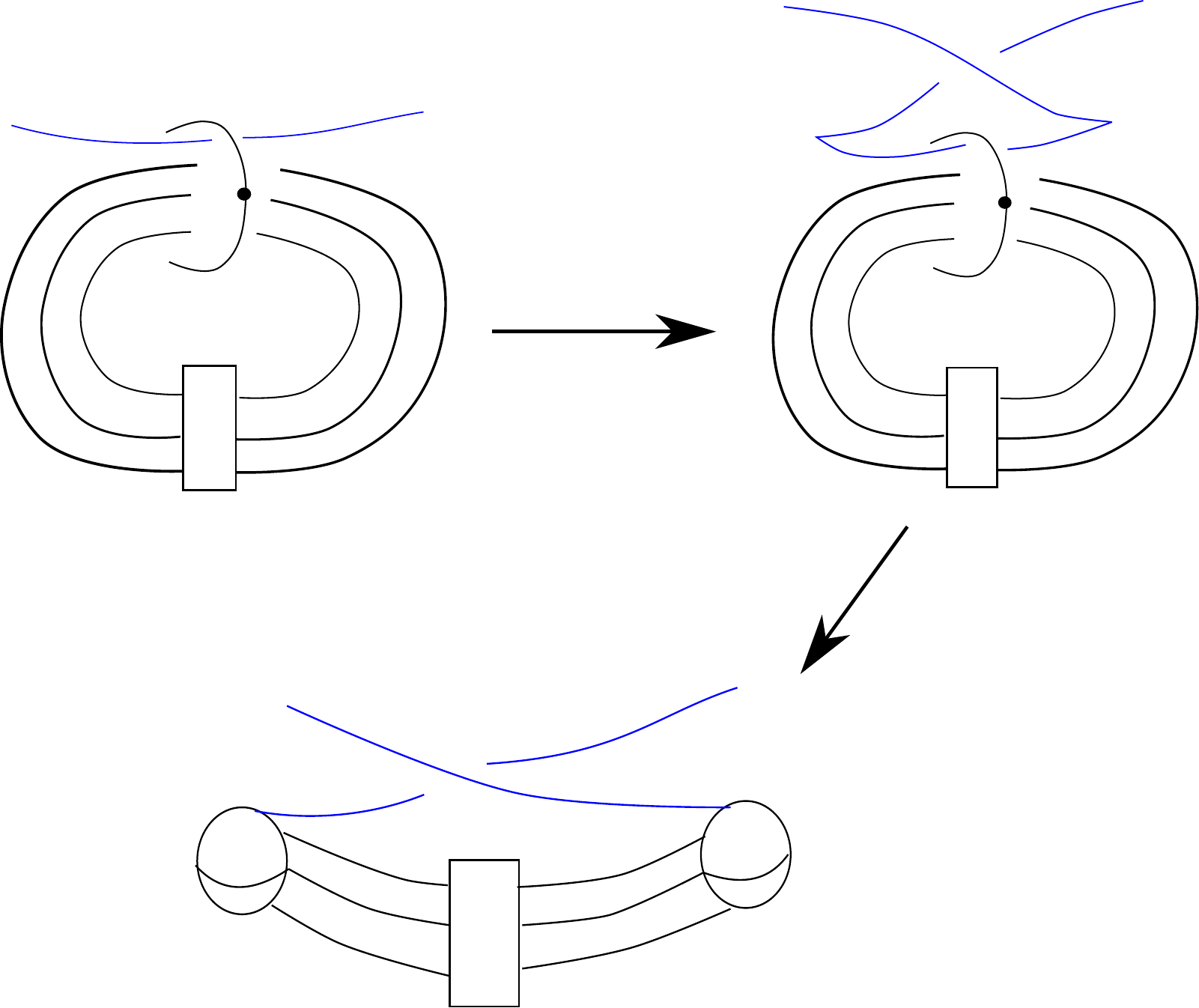}
  \put(120,160){\tiny{Isotopy}}
  \put(210,90){\tiny{Legendrian realization}}

\end{overpic}

\caption{
The contact framing of blue knot increased by +1 after a cork-twist.}
\label{tb}
\end{center}
\end{figure}
But this process does not change the smooth surgery coefficient. Let us consider the cobordism $X_2$ obtained by attaching a suitable number of corks to $X_1$ so that the manifold $X_2'$ obtained by applying the cork twists is Stein. The manifold $X_2$ and $X_2'$ are homeomorphic as the cork-twist homeomorphism can always be extend as homeomorphism on the 4-manifold by the result of Freedman~\cite{freedman}. Observe $b_2(X_2')=b_2(X_1)=b_2(X)$. Let $Y'$ be the top boundary of $X_2'$. Then there is a homology ribbon cobordism $M_3$ from $Y_1$ to $Y'$ which is given by attaching the above corks to the top of $X_1$, see the fourth picture in Figure~\ref{cob}. Glue this cobordism on top of $M_2$ to get our desired ribbon rational homology cobordism $M= M_2\cup M_3$ from $Y$ to $Y'$ with $Y'$ Stein fillable.
\end{proof}

Now we begin the proof of Theorem~\ref{AQ2} that says that given a compact 4-manifold with $\Q HS^3$ boundary one can construct a Stein 4-manifold with same algebraic topology.

\begin{proof}[Proof of Theorem~\ref{AQ2}]

Let $X$ is a compact manifold with connected boundary $Y$ which is a $\Q HS^3$, then we consider a handle decomposition of $X= X_0 \cup X_1 \cup X_2 \cup X_3$ where $X_i$ contains handles of index $i$. Consider the minimum set of 1-handles which generate the free part of $(H_1(X;\Q))$. Let $\bar{X}$ be the manifold obtained from $X$ by doing surgery on those 1-handles. (In Kirby calculus this is equivalent of replacing those dotted 1-handles with 0-framed unknotted 2-handles.) We will now show that this surgery operation does not change the $b_2$ (or more precisely the intersection form). As $Y$ is a $\Q HS^3$, $H_1(\bar{X}; \Q) = 0 = H_3 (\bar{X};\Q)$. But we are not doing anything with the 3-handles of $X$, so the only way the third homology of $\bar{X}$ vanishes with $\Q$ co-efficients is if the 3-handles cancel the 2-handles in homology. And thus from cellular homology, we can see that $b_2(X)=b_2(\bar{X})$. Also notice that the above surgery does not change the non-torsion elements of $H^2(X;\Z)$ so they have the same intersection form. Now apply the proof of Theorem~\ref{cstein} on $\bar{X}$ and we get the desired Stein manifold $X'$ with boundary $Y'$.
\end{proof}

\bibliography{references}
\bibliographystyle{plain}

\end{document}